\documentclass{amsart}
\usepackage{amsfonts}

\newtheorem{theorem}{Theorem}
\theoremstyle{plain}

\newtheorem{definition}{Definition}

\newtheorem{lemma}{Lemma}

\numberwithin{equation}{section}

\begin{document}
\title[H\`{a}jek - R\'{e}nyi's inequality and the SLLN]{On H\`{a}jek - R\'{e}nyi type inequality and application}
\author{Cheikhna Hamallah Ndiaye}
\email{chndiaye@ufrsat.org}
\address{LERSTAD, Universit\'{e} de Saint-Louis.}
\address{LMA - Laboratoire de Math\'ematiques Appliqu\'ees \\
Universit\'e Cheikh Anta Diop BP 5005 Dakar-Fann S\'en\'egal}
\author{Gane Samb LO}
\email{gane-samb.lo@ugb.edu.sn}
\address{LSTA, UPMC, FRANCE and LERSTAD, Universit\'{e} de Saint-Louis, SENEGAL.}

\begin{abstract}
\large
In this paper, we establish a new H\`{a}jek - R\'{e}nyi's type inequality
and obtain strong law of large numbers (SLLN) for arbitrary random
variables. We base our methods on demimartingales and convex functions
techniques. We obtain wide extensions of avalaible results, particularly those of Amini et al. \cite{amini06} and Rao \cite{rao1} and related SLLN's.
\end{abstract}

\maketitle

\Large

\section{INTRODUCTION}

\label{sec1}

This paper is concerned with general strong laws of large numbers (SLLN)
based on H\`{a}jek - R\'{e}nyi's type inequalities and therein provides sharp
generalization of recent forms of such inequalities. To begin with, we
remind the seminal result of H\`{a}jek - R\'{e}nyi \cite{hajek}. Let $%
X_{1},X_{2},...$ \ denote a sequence of random variables defined on a
fixed probability space $(\Omega ,\mathcal{F},P)\ ,S_{o}=0,S_{n}=\sum_{i=1}^{n}X_{i}$
for $n\geq 1.$

\begin{lemma} (Inequality of H\`{a}jek - R\'{e}nyi). If $\{X_{n},n\geq 1\}$ is a
sequence of independent random variables with mean zero, and $(b_{n})_{n\geq
1}$ is a nondecreasing sequence of positive real numbers, then for any $%
\varepsilon >0$ and any positive integer $m\leq n$,
\begin{equation*}
\mathbb{P}({\max_{m\leq k\leq n}}\frac{{\sum_{j=1}^{k}}X_{j}}{b_{k}}%
>\varepsilon )\leq \varepsilon ^{-2}{\sum_{j=m+1}^{n}}\mathbb{E}(\frac{%
X_{j}^{2}}{b_{j}^{2}})+\frac{1}{b_{m}^{2}}{\sum_{j=1}^{m}}\mathbb{E(}%
X_{j}^{2}).
\end{equation*}
\end{lemma}

\bigskip

\noindent This inequality has been studied and generalized by many authors. The latest
literature is given by Liu et al. \cite{liu} for negative association random
variables, Christofides \cite{christophe00} for demimartingales,
Christofides and Vaggelaton \cite{christophe04} for positive and negative
association, Rao\cite{rao3} and Sung \cite{sung} and Hu \textit{et al.} \cite{hu09} for
associated random variables and Wang \textit{et al.} \cite{wang10} for demimaringales
which extended \cite{christophe00}. Recent developments on the topic can be found
in Rao \cite{rao}. It is also worth mentionning that highly sophisticated
forms of such inequalities and related SLLN's are available for random
fields (see Nosz\'{a}ly et al. \cite{noszaly}, Farzkas et al. \cite
{fazekas99}, Hamallah et al. \cite{hamallah}, etc.).\newline

\noindent Our departure point\noindent\ is the H\`{a}jek - R\'{e}nyi type
inequality of Amini et al. \cite{amini06}, established for arbitrary random
variables when the second moments of the $X_{n}$'s are finite. Such results do not apply to stable random variables of parameter $1 \leq \alpha <2$ for example.\newline

\bigskip

\noindent This motivated us to extend the preceding H\`{a}jek - R\'{e}nyi type of
inequalities to a more general one. We would get more general
SLLN's and would be able to handle cases the former could not do, as the one mentioned above.\\

\noindent To achieve our objectives we call on demimartingales techniques and convex functions properties and provide a generalization of the inequality of Rao \cite{rao1} as a basis of our main inequality result in Theorem \ref{theo} below. Next, we derive from it our general SLLN in Theorem \ref{theo2}, which finally resulted in a considerable generalizations of the SLLN's of \cite{amini06}.\newline

\noindent The paper is organized as follows. In Section \ref{sec2}, we present some
definitions and lemmas which we need to prove our main result. In the section \ref{sec3},
we establish a H\`{a}jek - R\'{e}nyi's inequality for arbitrary random
variables. As a consequence, we obtain a strong law of large numbers for
arbitrary random variables. We conclude the paper by applications and comparison remarks with former results.
\newline

\noindent Whenever needed, any sequence of random variables used in the paper may be required to be $\mathcal{F}_{n}$%
-adapted where $(\mathcal{F}_{n},n\geq 0)$ is some filtration on $(\Omega ,\mathcal{F},P)$. We also
adopt the classical notation $X_{n}^{+}=max(0,X_{n})$ and $\ X_{n}^{-}=max(0,-X_{n}),$
$n=1,2,...$

\section{Definitions and some reminders} \label{sec2}

We begin to remind the notion of demimartingale introduced by Newmann and Wright \cite{wright}.

\begin{definition}
A $\mathcal{F}_{n}$-adapted sequence of random variables $\{T_{n},n\geq 1\}$ \ in $L^{1}(\Omega ,\mathcal{F},\mathbb{P})$ is called a demimartingale  if, for every $j\geq 1$, for every componentwise nondecreasing
function $g(x)$ of $x \in \mathbb{R}^{j}$,
\begin{equation}
\mathbb{E} \left( (T_{j+1}-T_{j})\ g(T_{1},...,T_{j}) \right) \geq 0. \label{dm1}
\end{equation}

\noindent If further the function $g$ in (\ref{dm1}) are required to be nonnegative (resp. non positive), then the sequence is a called a demisubmartingale (resp., a demisupermartingale).
\end{definition}

\noindent Let us make some immediate remarks. First a $\mathcal{F}_{n}$-martingale $T_{n}$ is a demimartingale. This can be seen by noting that
$$
\mathbb{E}\biggl((T_{j+1}-T_{j})g(T_{1},...,T_{j})\biggr)
=\mathbb{E}\biggl[\mathbb{E}\biggl((T_{j+1}-T_{j})g(T_{1},...,T_{j})\biggr)|
\mathcal{F}_{j}\biggr]
$$
$$
=\mathbb{E}\biggl[g(T_{1},...,T_{j})\mathbb{E}\biggl((T_{j+1}-T_{j})|%
\mathcal{F}_{j}\biggr)\biggr] =0
$$

\noindent by using the martingale property of the process $\{T_{n}, \mathcal{F}_{n},n\geq 1\}$. Similarly it can be seen that every $\mathcal{F}_{n}$-submartingale is a demisubmartingale. However a demisubmartingale need not be a $\mathcal{F}_{n}$-submartingale as showed by \cite{hadj}, even the filtration is the natural one.\\

\noindent We now recall some connexions between convex functions and demimartingales.
\begin{lemma} (see Lemma 2.1 and Corollary 2.1 in \cite{christophe00})\\
\noindent i) Let  $T_{1},T_{2},...,$ be a demisubmartingale (or a demimartingale) and
$\phi$ a nondecreasing convex function such that $\phi(T_{i})\in L_{1},\
i\geq 1.$ Then \ $\phi (T_{1}),\phi (T_{2}),...,$ is a demisubmartingale.\\

\noindent ii) If $\{T_{n},n\geq 1\}$ \ is a demimartingale, then $\{T_{n}^{+},n\geq 1\}$
 is a demisubmartingale et $\{T_{n}^{-},n\geq 1\}$  is a demisubmartingale.
\end{lemma}

\noindent Rao introduced the following inequality that we are going to generalize.
\begin{lemma}(see Rao \cite{rao1}). Let  $\{T_{n},n\geq 1\}$ be a demisubmartingale and
$\phi(.)$ be a nonnegative nondecreasing convex function such that $\phi(T_{o})=0$. Let  $\chi(b)$  be a positive nondecreasing function of  $b>0$ and let $0=b_{o}<b_{1}\leq ...\leq b_{n}$. Then
$$
\mathbb{P}\left( \phi(T_{k}) \leq \chi(b_{k}), \text{ } 1 \leq k \leq n \right)
\geq 1-\sum_{k=1}^{n} \frac{\mathbb{E}[\phi (T_{k})]-\mathbb{E}[\phi (T_{k-1})]}{\chi (b_{k})}.
$$
\end{lemma}

\noindent We also need this result to derive general SLLN's.
\begin{lemma} \label{lemma3} (See Theorem 2.4\ in T\'{o}m\'{a}cs -Libor \cite{thomas06}). Let
$\{\alpha _{k},\ k\in N\}$ be a sequence of non-negative real numbers, $r>0$, and $\{b_{k},\ k\in N\}$ a nondecreasing unbounded sequence of positive real numbers. Assume that $\sum_{k=1}^{\infty}\alpha _{k}\ b _{k}^{-r}<\infty $
and that there exists  $c>0$ such that for any  $n\in N$ and any  $\varepsilon >0$
\begin{equation*}
\mathbb{P}({\max_{k\leq n}}|S_{k}|\ \geq \varepsilon )\leq c\varepsilon ^{-r}%
{\sum_{k=1}^{n}}\alpha _{k}.
\end{equation*}
Then
$$
\lim_{n\rightarrow +\infty }\frac{S_{n}}{b_{n}}=0, \text{ } a.s
$$
\end{lemma}
\section{OUR RESULTS} \label{sec3}
Before we state our results, let us make this simple remark : there exist two demisubmartingales
\begin{equation}
\{u_{n}, \text{ } n \geq 1\} \text{ and } \{v_{n}, \text{ } n \geq 1\} \label{maj1}
\end{equation}

\noindent such that
\begin{equation}
\forall n\geq 1, \text{  } S_{n}\leq u_{n}+v_{n} \text{  } a.s. \label{maj2}
\end{equation}

\noindent To see that it suffices to set $u_{n}=\sum_{i=1}^{n}X_{i}^{+}$  and  $v_{n}=\sum_{i=1}^{n}X_{i}^{-}$. Then (\ref{maj2}) is evident. Finally, each of these sequence is a $\mathcal{F}_{n}$-submartingale and then a demisubmartingale.\\

\noindent We use the demisubmartingales $u_{n}$ and $u_{n}$ in our general H\`{a}jek - R\'{e}nyi's type of inequality below.

\begin{theorem} \label{theo} Let  $\{X_{n},n\geq 1\}$  be an arbitrary sequence of $a.s$ finite random variables. Let $\phi(.)$ be a nonnegative nondecreasing convex function such that $\phi (S_{o})=0$  and there exist a positive number real $K$ such that
for all  $(x,y) \in \mathbb{R}^{2}$,
\begin{equation*}
\phi (x+y)\leq K[\phi (x)+\phi (y)]
\end{equation*}

\noindent Let $\chi (b)$  be a positive non decreasing function of $b>0$ and $0=b_{0}<b_{1}\leq b_{2}\leq ... \leq b_{n}$ and finally set
$A_{n}=\left( \phi (S_{k})\leq \chi (b_{k}),1\leq k\leq n \right)$, $n \geq 1$. Then for any $n \geq 1$,
\begin{equation*}
\mathbb{P}(A_{n})\geq 1-2K\ \displaystyle{\sum_{k=1}^{n}}\ \displaystyle%
\frac{\mathbb{E}[\phi (u_{k})+\phi (v_{k})]-\mathbb{E}[\phi (u_{k-1})+\phi
(v_{k-1})]}{\chi (b_{k})}
\end{equation*}
whenever the right member is well-defined and where $\{u_{k},k\geq 1\}$  and  $\{v_{k},k\geq 1\}$ are two demisubmartingales defined in (\ref{maj1})
\end{theorem}

\begin{proof} By (\ref{maj2}),
\begin{equation*}
S_{k}\leq u_{k}+v_{k},\ 1\leq k\leq n, \text{ } a.s
\end{equation*}

\noindent and then
\begin{equation*}
\phi (S_{k})  \leq \phi (u_{k}+v_{k}) \leq K[\phi (u_{k})+\phi (v_{k})]\text{ }a.s.
\end{equation*}

\noindent Next,
$$
\mathbb{P}(A_{n}^{c})  =\mathbb{P}\biggl(\max_{1\leq k<n} \frac{\phi(S_{k})}{\chi (b_{k})}\geq 1\biggr)
$$
$$
\leq \mathbb{P}\biggl(\max_{1\leq k<n} \frac{K\phi (u_{k})}{\chi (b_{k})}\geq 1/2\biggr)
+\mathbb{P}\biggl(\max_{1\leq k<n}\ \frac{K\phi (v_{k})}{\chi (b_{k})}\geq 1/2\biggr)
$$
$$
\leq 2K \sum_{k=1}^{n} \frac{\mathbb{E}[\phi(u_{k})+\phi (v_{k})]-\mathbb{E} \lbrack \phi (u_{k-1})+\phi (v_{k-1})]}{\chi (b_{k})},
$$

\noindent whenever the right member is well-defined.\\
\end{proof}

The above result is a generalization of the result of Rao \cite{rao1} established for demisubmartingales to a sequence of arbitrary random variables.\\

\noindent

\begin{theorem} \label{theo2}
Let  $\{X_{n},n\geq 1\}$  be an arbitrary $a.s.$ finite random variables and
\ $\phi (.)$  be a nonnegative nondecreasing convex function such that  $\phi(0)=0.$ Let $\chi (b)$ be a positive nondecreasing function of $b>0$ such that  $\chi (b)\rightarrow \infty$ as $b\rightarrow\infty$ and let $(b_{n})_{n\geq 1}$ a nondecreasing and unbounded sequence of positive real numbers. Further suppose the following series is well-defined and satisfies
\begin{equation*}
\sum_{k=1}^{\infty } \frac{\mathbb{E}[\phi
(u_{k})+\phi (v_{k})]-\mathbb{E}[\phi (u_{k-1})+\phi (v_{k-1})]}{\chi (b_{k})}%
<\infty ,
\end{equation*}
where  $u_{k}$ and $v_{k}$ are demisubmartingales defined in \ref{maj2}.\newline
Then
\begin{equation*}
\frac{\phi (S_{n})}{\chi (b_{n})} \rightarrow 0, \text{ a.s. as } n\rightarrow \infty.
\end{equation*}
\end{theorem}

\begin{proof}
The result follows immediately by Lemma \ref{lemma3}. We omit the details.
\end{proof}

\section{Applications and comparison} \label{sec4}

\subsection{Comparison with former results}
Our result is to be compared with the following results of Amini et al. \cite{amini06} for arbitrary random variables.

\begin{theorem}
\label{theo1} \label{theo1} (See Amini and Bozorgnia, A .\cite{amini06}) Let
$\{X_{n},n\geq 1\}$ be a sequence of random variables with $\mathbb{E}(X_{n})=0$, $%
\sigma _{n}^{2}=Var(X_{n})=\mathbb{E}(X_{n}^{2})<\infty ,\ n\geq 1$ , and $\{b_{n},\
n\geq 1\}$ be a sequence of positive nondecreasing real numbers, then for
every $\varepsilon >0$
\begin{equation*}
\mathbb{P}({\max_{1\leq k\leq n}}\frac{|S_{k}|}{b_{k}}\geq \varepsilon )\leq
\frac{8}{\varepsilon ^{2}}{\sum_{k=1}^{n}}\frac{\sigma _{k}^{2}}{b_{k}^{2}}+2%
{\sum_{k=2}^{n}}\frac{\sigma _{k}{\sum_{i=1}^{k-1}\sigma _{i}}}{b_{k}^{2}}.
\end{equation*}
\end{theorem}

\noindent This result is derived from Theorem \ref{theo} by putting $\phi(x)=|x|^{2}$ and $\chi (b)=\varepsilon b$ with $E(X_{n})=0$, $\mathbb{E}(X_{n}^{2})<\infty$ for all $n\geq 1$, $u_{n}=\sum_{i=1}^{n}X_{i}^{+}$  and $v_{n}=\sum_{i=1}^{n}X_{i}^{-}$.\\

\noindent More generally, our results apply for $\phi(x)=|x|^{\nu}$ for $\nu \geq 1$. This means that we are able to have SLLN's when $\nu=1$ that is when the $X_{n}$'s only have finite first moments.\\

\noindent For example, for a sequence of random variables with stable laws of parameters $1\leq \alpha_{n} <2$, the second moments are infinite for those random variables while their first moments exists. The results of Amini et al. do not apply.\\

\subsection{Comparison with other inequalities or SLLN's}
We rediscover Theorem 3.12 of T\'{o}m\'{a}cs \textit{et al.} \cite{thomas06} and Theorem 2.2 of Christophides
\cite{christophe00} for demimartingale sequence $\{S_{n},n\in \mathbb{N}\}$ by taking $\phi
(x)=(x)^{r}$ for all $x$ nonnegative real and $0$ otherwise, $u_{n}=(S_{n}^{+})^{r}$ \ et \ $%
v_{n}=(S_{n}^{-})^{r}$ \ for all $r\geq 1$ which are two demisubmaringales
and remarking that $|S_{n}|^{r} =(S_{n}^{+})^{r}+ (S_{n}^{-})^{r}$.\\

\noindent Our main result also extends Theorem 2.1 and related results of Wang \textit{et
al.} \cite{wang10} in the case where the function $g$ defined in this above paper
is in addition nondecreasing and there exist a positive number real $K$ such that
for all  $(x,y) \in \mathbb{R}^{2}$,
\begin{equation*}
 g(x+y)\leq K[g(x)+g(y)]
\end{equation*} for demimartingales random sequences to
arbitrary integrale random variables.It is the case for $g(x)=(x)^{\nu}$ for all $x$ nonnegative real and $0$ otherwise, for $\nu \geq 1$ \\

\noindent Let us mention that the inequality of Laha and Rohtgi \cite{laha} proved for submartingales could be extended to a version of our Theorem \ref{theo} for $\phi(x)=|x|^{\nu}$, for $\nu \geq 1$.


\begin{thebibliography}{99}
\bibitem{amini06}  Amini, M. and Bozorgnia, A.(2006). Some maximal
inequality for random variables and applications. \textit{Iranian Journal of Science
et Technology, Transaction A}. (30), A3.

\bibitem{christophe00}  Christofides, T.C. (2000). Maximal inequalities for
demimartingales and a strong law of large numbers. \textit{Statistics and
Probability letters}, 50, 357 - 363.

\bibitem{christophe04}  Christofides, T.C and Vaggelation, E. (2004). A
connection between supermodular ordering and positive / negative
association. \textit{Journal of Multivariate analysis}, 88, 138 - 151.

\bibitem{fazekas}  Fazekas, I and Klesov, O., A general approach to the
strong laws of large numbers. \textit{Theory of Probab. Appl.}, 45/3 (2000) 568 - 583.

\bibitem{fazekas99}  Fazekas, I., Klesov, O. and Nosz\'{a}ly, Cs., T\'{o}m%
\'{a}cs, T.(1999). Strong laws of large numbers for sequences and fields.
(\textit{Proceedings of the third Ukrainian - Scandinavian. Conference in
Probability theory and Mathematical Statistics}. 8 - 12. June 1999 Kyiv,
Ukraine Theory Stoch. Process. 5, no. 3-4, 91-104. (MR2018403)

\bibitem{hadj}  hadjikyriakou, M. (2010). \textit{Probability and Moment Inequalities for
Demimartingales and Associated Random Variables}. Ph. D. Dissertation,
University of Cyprus, Nicosia.

\bibitem{hamallah}  On general strong laws of large numbers for fields of
random variables. \textit{Ann. Math. Inform.}, \textbf{38}, 3-13.

\bibitem{hajek}  H\`{a}jek, J. and R\'{e}nyi, A. (1955). A generalization
of an inequality of Kolmogorov. \textit{Acta. Math. Acad. Sci. Hung.}, 6, 281 - 284.

\bibitem{hu09}  Hu, S. H., Wang, X. J., Yang, W.H. and Zhao, T. (2009). The H%
\`{a}jek-R\'{e}nyi-type inequality for associated random variables.
\textit{Statistics and probability letters}, 79, 884 - 888.

\bibitem{laha}  Laha, R. G. and Rohtgi, V.K. (1979). \textit{Probability theory}.
John Wiley and Sons. New York.

\bibitem{liu}  Liu, J., Gan, S. and Chen, P. (1999). The H\`{a}jek-R\'{e}nyi
inequality for the NA random variables and its applications. \textit{Statistics and
Probability letters}, 43, 99 - 105.

\bibitem{wright}  Newman, C.M and Wright, A.L. (1982) Associated random variables
and martignale inequalities. \textit{Z. Wahrsch. Theorie und Verw. Gebiete}, 59, 36 -
371.

\bibitem{noszaly}  Nosz\'{a}ly, Cs. and T\'{o}m\'{a}cs, T.(2000). A general
approach to strong laws of large numbers for fields of random variables.
\textit{Ann. Univ. Sci. Budapest}. 43, 61 - 78. 

\bibitem{rao1}  Rao, P. B.L.S. (2002) Whittle type inequality for
demisubmartingales, \textit{Proc. Amer. Math. Soc.}, 130, 3719 - 3724.

\bibitem{rao} Rao, P. B. L. S.(2012). Associated Sequences
Demimartingales and Nonparametric Inference. Birh\H{a}user. (Doi
10.1007/79-3-0348-0240-6).

\bibitem{rao3} Rao, P.B.L.S.(2002). H\'{a}j\`{e}k - R%
\'{e}nyi inequality for associated sequences. \textit{Statistics and Probability Letters}, 57, 139-144.

\bibitem{sung} Sung, Soo Hak.(2008). A note on the H\'{a}j\`{e}k - R%
\'{e}nyi inequality for associated random variables. \textit{Statistics and Probability Letters},
78, 885-889.


\bibitem{thomas06}  T\'{o}m\'{a}cs, T. and Libor, ZS., A H\'{a}j\`{e}k - R%
\'{e}nyi type inequality and its applications. \textit{Annales Mathematicae and
Informaticae}, 33 (2006) 141 - 149.

\bibitem{wang10}  Wang Xuejun, Hu Shuhe, Zhao Ting, and Yang Wenzhi. (2000)
Doob's Type Inequality and Strong law of Large Numbers for Demimartingales. \textit{Journal of Inequalities and Applications}, vol. 2010
article ID. 838301, 11p,doi:10.1155/2010/838301, 2010.
\end{thebibliography}
\end{document}